\documentclass[11pt]{amsart}
\usepackage{amsmath,amsthm,amssymb,amsfonts,graphicx,mathrsfs,bm,amscd,latexsym,caption}
\usepackage{color,enumitem,hyperref}
\usepackage{blindtext}
\usepackage{textcomp}
\usepackage[all]{xy}

\newtheorem{theorem}{Theorem}[section]
\newtheorem{proposition}[theorem]{Proposition}
\newtheorem{lemma}[theorem]{Lemma}

\def\R{\mathbb{R}}

\def\N{\mathbb{N}}

\def\e{\epsilon}
\def\g{\gamma}
\def\d{\delta}

\def\G{\Gamma}
\def\a{\alpha}

\newcommand{\whocare}[1]{}

\DeclareMathOperator{\diam}{diam}
\DeclareMathOperator{\w}{{\bf w}}
\DeclareMathOperator{\sdim}{s-dim}
\DeclareMathOperator{\dist}{dist}

\DeclareMathOperator{\Lip}{Lip}

\numberwithin{equation}{section}

\title{Parametrizability of infinitely generated attractors}

\author{Eve Shaw}

\author{Vyron Vellis}

\thanks{V.V. was partially supported by NSF DMS grants 1952510 and 2154918.}
\date{\today}

\address{Department of Mathematics\\ The University of Tennessee\\ Knoxville, TN 37966}
\email{mshaw20@vols.utk.edu}
\address{Department of Mathematics\\ The University of Tennessee\\ Knoxville, TN 37966}
\email{vvellis@utk.edu}

\begin{document}

\maketitle

\begin{abstract}An infinite iterated function system (IIFS) is a countable collection of contraction maps on a compact metric space. In this paper we study the conditions under which the attractor of such a system admits a parameterization by a continuous or H\"older continuous map of the unit interval.
\end{abstract}

\section{Introduction}

Iterated function systems are among the most standard and canonical methods in mathematics of producing fractal sets. An \emph{iterated function system} (abbv. IFS) is a finite collection $\mathcal{F}$ of contraction maps on a complete metric space $X$. Hutchinson \cite{Hutchinson} showed that for each IFS $\mathcal{F}$, there exists a unique nonempty compact set $K\subset X$ (called the  \emph{attractor} of $\mathcal{F}$) such that $K = \bigcup_{\phi\in\mathcal{F}}\phi(K)$. The \emph{similarity dimension} of an IFS $\mathcal{F}$ is the unique solution to the equation
\begin{equation}\label{eq:sdim}
 \psi_{\mathcal{F}}(t) := \sum_{\phi \in \mathcal{F}}\Lip(\phi)^t = 1
 \end{equation}
where $\Lip(\phi)$ denotes the infimum of all $L>0$ for which $\phi$ is $L$-Lipschitz. Here and for the rest of the paper we only consider nondegenerate proper contractions, that is, we always assume that $\Lip(\phi) \in (0,1)$. 

The connection between the similarity dimension of an IFS $\mathcal{F}$ and the Hausdorff dimension of its attractor $K$ was established by Hutchinson \cite{Hutchinson} who showed that if $\mathcal{F}$ is an IFS of similarities on $\R^n$ satisfying the open set condition, then $\dim_H(K) = \sdim(\mathcal{F})$. Recall that a similarity in $\R^n$ is the composition of a dilation and an isometry. An IFS  $\mathcal{F}$ on $\R^n$ satisfies the \emph{open set condition} (abbv. OSC) if there exists a nonempty open set $U \subset \R^n$ such that $\phi(U) \subset U$ for all $\phi \in \mathcal{F}$, and $\phi(U)\cap \phi'(U) = \emptyset$ for all distinct $\phi,\phi' \in \mathcal{F}$. Many well known fractals (such as the standard Cantor set, the Sierpi\'nski carpet, the von Koch snowflake, etc.) are attractors of IFS of similarities on the plane with the OSC. See also \cite{Sch94, Sch96,FF,FHOR} for the necessity of the OSC. 

A natural question in the theory of Dynamical Systems is the regularity of an IFS attractor and whether it admits ``good'' parameterizations by the unit interval. Hata \cite{Hata} showed that if the attractor $K$ of an IFS is connected, then it is the image of a curve, that is, the image of $[0,1]$ under a continuous map. The second named author and Badger \cite{BV} improved Hata's result by proving that if the attractor $K$ of an IFS $\mathcal{F}$ is connected, then it is the image of $[0,1]$ under a $\frac1{\a}$-H\"older continuous map for any $\a > \sdim(\mathcal{F})$. Under the extra assumptions that $X=\R^n$ and that $\mathcal{F}$ is an IFS of similarities on $\R^n$ satisfying the OSC, Remes \cite{Remes} showed earlier that the attractor is the image of $[0,1]$ under a $\frac1{\a}$-H\"older continuous map where one can actually have $\a = \sdim(\mathcal{F})$. Remes' result is sharp in that there exists no $\frac1{\a}$-H\"older parameterization if $\a<\sdim(\mathcal{F})$. The assumption $X=\R^n$ in Remes' theorem can be replaced by the assumption $\mathcal{H}^{\sdim(\mathcal{F})}(K) >0$ \cite{BV}. Here and for the rest of the paper $\mathcal{H}^\a$ denotes the Hausdorff $\a$-dimensional measure.

In their celebrated paper, Mauldin and Urba\'nski \cite{MU} (see also \cite{M}) further extended Hutchinson's theory and introduced the notion of an \emph{infinite iterated function system} (abbv. IIFS); i.e., an infinite countable collection of contractions on a compact metric space $X$. Ever since their introduction, IIFSs have played a major role in fractal geometry, geometric group theory, and number theory; see \cite{MU99,HU02,MU02,UZ02,KZ06,MSU09,JR12,SW15,RGU16,BF23} and the references therein.

Here, unlike in most literature, we do not assume that contractions are conformal or even injective. Given an IIFS $\mathcal{F}=\{\phi_i: i\in\N\}$ on a compact metric space $X$, we define the attractor of $\mathcal{F}$ by
\[ K = \bigcup_{(i_n) \subset \N} \bigcap_{n=1}^{\infty} \phi_{i_1}\circ \cdots\circ \phi_{i_n} (X).\]
The attractor $K$ in the infinite setting may not be compact. Moreover, although $K = \bigcup_{\phi\in\mathcal{F}}\phi(K)$, there may exist multiple nonempty subsets of $X$ with this property. If, additionally, each $\phi_i\in \mathcal{F}$ is injective, and if each $x\in X$ is contained in at most finitely many $\phi_i(X)$, then
\begin{equation}\label{eq:IIFSattractor} 
K = \bigcap_{n\in\N} \bigcup_{i_1,\dots,i_n \in \N} \phi_{i_1}\circ \cdots\circ \phi_{i_n} (X).
\end{equation}

In the infinite setting, the auxiliary function $\psi_{\mathcal{F}}$ in \eqref{eq:sdim} is either infinite for all $t>0$, or it is continuous and strictly decreasing on an interval $(a,\infty)$ for some $a\geq 0$. Hence, unlike in the finite case, equation \eqref{eq:sdim} may not have a solution. We define the similarity dimension of an IIFS $\mathcal{F}$ as
\[ \sdim(\mathcal{F}) := \inf \{t \geq 0 : \psi_{\mathcal{F}}(t) \leq 1\} .\]
By Fatou's Lemma, the infimum above is in fact a minimum. As in the finite case, if $\mathcal{F}$ is an IIFS of similarities on $X=\overline{U}\subset\R^n$ where $U$ is a bounded domain, satisfying the OSC, then the Hausdorff dimension of the attractor $K$ is equal to $\sdim(\mathcal{F})$ \cite[Corollary 3.17]{MU}.

The purpose of this paper is to study the parameterizability of IIFS attractors. In the infinite case an interesting dichotomy appears. On the one hand, in Proposition \ref{prop:s=1} we show that if the attractor $K$ of an IIFS $\mathcal{F}$ is a continuum and if $\sdim(\mathcal{F}) = 1$, then $K$ is a line segment; see also \cite{MVU01} for a similar phenomenon. On the other hand, both Hata's theorem and the Badger-Vellis theorem are false when the similarity dimension is greater than 1, even if it is arbitrarily close to 1.

\begin{theorem}\label{thm:1}
\begin{enumerate}
\item For each $\e>0$, there exists an IIFS of similarities on the unit square $[0,1]^2$ having the OSC such that $\sdim(\mathcal{F}) < 1+\e$, and its attractor is a continuum but not path connected.
\item For each $\e>0$, there exists an IIFS of similarities on the unit square $[0,1]^2$ having the OSC such that $\sdim(\mathcal{F}) < 1+\e$, and its attractor is the image of a curve, but not the image of a H\"older curve.
\end{enumerate}
\end{theorem}

While the first example may not be too surprising, the second example has the additional property that for every two points there exists a Lipschitz curve in the attractor that connects them. A common theme in both these examples is the existence of a ``bad 1-skeleton'' inside the attractor which is not the image of a curve (in the first case) or not the image of a H\"older curve (in the second case). If such bad structures are absent, we show that the attractor admits good parameterizations.

\begin{theorem}\label{thm:3}
Let $\mathcal{F} = \{\phi_i\}_{i\in\N}$ be an IIFS on a compact metric space so that the attractor $K$ is compact, $\lim_{i\to\infty}\Lip(\phi_i)=0$, and there exists a curve $\gamma:[0,1] \to K$ whose image intersects $\phi_i(K)$ for all $i$.
\begin{enumerate}
\item The attractor $K$ is the image of a curve.
\item If $\gamma$ is $\frac1{s}$-H\"older for some $s\geq 1$, then for each $\a>\max\{s,\sdim(\mathcal{F})\}$ the attractor $K$ is the image of a $\frac1{\a}$-H\"older curve.
\end{enumerate}
\end{theorem}

We leave it as an open question whether in the second part of the theorem one can choose $\a=\max\{s,\sdim(\mathcal{F})\}$; this is unknown even for IFSs. Furthermore, in the case that $\mathcal{F}$ is finite, the existence of the curve $\gamma$ in both parts of Theorem \ref{thm:3} is guaranteed by \cite[Theorem 1.1]{BV}. Finally, the condition $\lim_{i\to\infty}\Lip(\phi_i)=0$ is necessary for Theorem \ref{thm:3}; see \textsection\ref{sec:ex3}.

The construction of the two examples of Theorem \ref{thm:1} is given in Section \ref{sec:ex} and we prove Theorem \ref{thm:3} in Section \ref{sec:param}.

\subsection{Symbolic notation}
Here and for the rest of the paper, given a countable (infinite or finite) set $A$ and an integer $n\geq 0$, we denote by $A^n$ the set of words formed from $A$ of length $n$, with the convention $A^0 = \{\varepsilon\}$ and $\varepsilon$ is the empty word. We denote $A^* = \bigcup_{n\geq 0}A^n$ and by $A^{\N}$, the set of infinite words formed with letters from $A$. Given $w = i_1 i_2 \cdots \in A^{\N}$ we denote $w(n) = i_1\cdots i_n$ the truncated sub-word of $w$.

Given an IFS or IIFS $\mathcal{F} = \{f_i : i\in A\}$ and a finite word $w = i_1\cdots i_n \in A^n$, we denote the length $n$ of $w$ by $|w|$ and
\[ f_w = f_{i_1} \circ \cdots \circ f_{i_n}.\]

Given quantities $x,y \geq 0$ and a constant $a>0$ we write $x \lesssim_{a} y$ if there exists a constant $C$ depending on at most $a$ such that $x \leq C y$. If $C$ is universal, we write $x\lesssim y$. We write $x\simeq_{a} y$ if $x\lesssim_{a}y$ and $y\lesssim_{a} x$. 

\section{Parametrizations of infinite IFS attractors}\label{sec:param}

In this section we prove Theorem \ref{thm:3}. We start by proving the simple fact that if an IIFS has similarity dimension equal to 1 and the attractor is a continuum, then the attractor is a line segment. This was shown by Hutchinson for finite IFS of similarities on Euclidean spaces \cite[Remark 3.4]{Hutchinson}.

\begin{proposition}\label{prop:s=1}
Let $K$ be the attractor of an infinite or finite IFS $\mathcal{F}$ on a compact space. If $K$ is a continuum and if $\sdim(\mathcal{F})=1$, then $K$ is isometric to a closed line segment.
\end{proposition}

\begin{proof}
Write $\mathcal{F} = \{\phi_i : i\in A\}$. We first claim that $\mathcal{H}^1(K) \leq \diam{K}$. Fix $\delta>0$ and let $n\in\N$ such that $(\diam{K})\Lip(\phi_w)<\delta$ for all $w\in \N^n$. Since $\psi_{\mathcal{F}}(1)\leq 1$,
\begin{align*} 
\mathcal{H}^1_{\delta}(K) \leq \sum_{w\in A^n}\diam{\phi_w(K)} &\leq \diam{K}\sum_{i_1,\dots,i_n\in A}\Lip{\phi_{i_1}}\cdots \Lip{\phi_{i_n}}\\ 
&= \diam{K}(\psi_{\mathcal{F}}(1))^n \leq \diam{K}
\end{align*}
and the claim follows by taking $\d\to 0$.
Hence, $K$ is the Lipschitz image of $[0,1]$ \cite[Theorem 4.4]{AO}. Fix $p,q\in K$ such that $d(p,q)=\diam{K}$ and fix an arc $\gamma \subset K$ with endpoints $p,q$. Note that
\[\diam{K} = \mathcal{H}^1(K) \geq \mathcal{H}^1(\gamma) \geq \diam{\gamma} = \diam{K}.\]
Therefore, $K=\gamma$. For each $x,y \in \gamma$ denote by $\g(x,y)$ the subarc of $\gamma$ with endpoints $x,y$.

Let $x,y \in \gamma$ such that $x$ is between $p$ and $y$. Then, 
\begin{align*} 
\diam{\gamma} &\leq d(p,x) + d(x,y) + d(y,q)\\
&\leq \mathcal{H}^1(\gamma(p,x)) + \mathcal{H}^1(\gamma(x,y)) + \mathcal{H}^1(\gamma(y,q))\\ 
&= \mathcal{H}^1(\gamma)\\
&= \diam{\gamma}.
\end{align*}

Therefore, $\mathcal{H}^1(\gamma(x,y)) = d(x,y)$ for all $x,y \in \gamma$, which yields that $K$ is isometric to the line segment $[0,\diam{K}]$.
\end{proof}

The rest of Section \ref{sec:param} is devoted to the proof of Theorem \ref{thm:3}. Henceforth, we assume that we have an IIFS $\mathcal{F}=\{\phi_1,\phi_2,\dots\}$ on a compact space $X$ so that $K$ is compact, $\lim_{i\to\infty}\Lip(\phi_i)=0$, and there exists a curve $\gamma :[0,1] \to K$ whose image intersect each set $\phi_i(K)$. We make some standard reductions. 

First, if $K$ is a point, then the claim of the theorem is trivial. Therefore, we may assume that $K$ is nondegenerate and, rescaling the metric, we may also assume that $\diam(K)=1$. Second, since $\lim_{i\to\infty} \Lip(\phi_i) =0$, we may assume that 
\[ \Lip(\phi_{1}) = \max_{i\in\N} \Lip(\phi_i) .\] 
Third, by traversing the image of $\g$ backwards if necessary, we may assume that $\gamma(0)=\gamma(1)$. For any point $p$ in the image of $\gamma$, by reparameterizing $\gamma$, we may assume that $\gamma(0)=\gamma(1)=p$. Moreover, for any $p,q$ in the image of $\gamma$, by reparameterizing $\gamma$, we may assume that there exists $[a,b] \subset [0,1]$ such that $\gamma(a)=p$ and $\gamma(b)=q$.

Here and for the rest of this section, given $w=i_1 i_2 \dots i_n \in\mathbb{N}^*$, we denote 
\[L_w=\Lip(\phi_{i_1})\Lip(\phi_{i_2})\dots\Lip(\phi_{i_n})\] 
with the convention $L_{\varepsilon}=1$. Note that in general, $L_w \geq \Lip(\phi_w)$.

\subsection{Path connectedness of IIFS attractors}\label{sec:thm3}

The first step in the proof of Theorem \ref{thm:3} is the following lemma which shows that $K$ is pathwise connected.

\begin{lemma}\label{lem:conn}
Let $w\in \N^*$ and $x,y \in \phi_w(K)$. There exists a continuous map $f: [0,1] \to \phi_w(K)$ such that $f(0) = x$ and $f(1)=y$. If $\gamma$ is $\frac{1}{s}$-H\"older for some $s>\sdim(\mathcal{F})$ and with H\"older constant $H_0$, then $f$ can be chosen to be $\frac{1}{s}$-H\"older with H\"older constant $H \lesssim_{\psi_\mathcal{F}(s),s,H_0,L_1} L_w$,. 
\end{lemma}

\begin{proof}
Clearly, we may assume that $x\neq y$. Moreover, it suffices to assume that $w=\varepsilon$, and that there is no $j\in \N$ so that $x,y\in\phi_j(K)$, as otherwise we could pass to the longest common word. By the Kuratowski embedding theorem, we may further assume that $K$ is a subset of $\ell_\infty$. 

We construct a sequence of continuous maps $(f_n : [0,1] \to \ell_{\infty})_{n\geq 0}$, sequences of finite collections of closed nondegenerate intervals $(\mathscr{B}_n)_{n\geq 0}$, $(\mathscr{E}_n)_{n\geq 0}$ in $[0,1]$, and an injection $\w: \bigcup_{n\geq 0}\mathscr{E}_n \to \N^*$ with the following properties.
\begin{enumerate}
\item[{(P1)}] For each $n\in \N$, intervals in $\mathscr{B}_{n}\cup \mathscr{E}_n$ intersect only at endpoints and the union of all these intervals is $[0,1]$.
\item[{(P2)}] For each $J \in \mathscr{E}_n$ there exists unique $S\in \mathscr{B}_{n+1}$ such that $S \subset J$. Conversely, for each $S \in \mathscr{B}_{n+1}\setminus \mathscr{B}_{n}$ there exists unique $J_S\in \mathscr{E}_n$ such that $S \subset J_S$. 
\item[{(P3)}] For any $n \geq 0$ and any $J \in \mathscr{E}_{n+1}$, there exists unique $J' \in \mathscr{E}_n$ such that $J \subset J'$. Moreover, there exists $u\in \N^*$ such that $|u|>0$ and $\w(J) = \w(J')u$.
\item[{(P4)}] For each $n\geq 0$ and $J \in \mathscr{E}_n$, there exist distinct $i,j \in \N$ such that $f_n|J$ is a linear map mapping the left endpoint in $\phi_{\w(J)i}(K)$ and the right endpoint in $\phi_{\w(J)j}(K)$. 
\item[{(P5)}] For each $n\geq 1$, and each $S\in \mathscr{B}_n$ there exists a closed nondegenerate interval $I$ and a linear map $\zeta: S\to I$ such that
\[f_{n+1}|S = f_{n}|S = \phi_{\w(J_S)} \circ (\gamma|I) \circ \zeta.\]
\item[{(P6)}] For each $n\geq 0$, $f_n(0) = x$ and $f_n(1)=y$. 
\item[{(P7)}] For each $n\geq 0$, and each $J \in \mathscr{E}_n$, $\|f_n-f_{n+1}\|_{J, \infty} \leq 2L_{\w(J)}$.
\end{enumerate}

Before the construction of $(f_n)_n$, $(\mathscr{B}_{n})_n$, $(\mathscr{E}_{n})_n$, and $\w$ we remark that (P3) and a simple induction yield that
\begin{enumerate}
\item[{(P8)}] For all $n\in\N$ and all $J \in \mathscr{E}_{n}$, $|\w(J)| \geq n$.
\end{enumerate}

The construction is done in an inductive fashion. Define $\mathscr{E}_0 = \{[0,1]\}$, $\mathscr{B}_0 = \emptyset$, $\w([0,1]) = \varepsilon$, and $f_0 : [0,1] \to \ell_{\infty}$ to be the linear map with $f_0(0) = x$ and $f_0(1)=y$. Property (P4) is immediate while the rest of the properties are vacuous.

Assume now that for some integer $n\geq 0$ we have defined a continuous map $f_n: [0,1] \to \ell_{\infty}$, collections $\mathscr{B}_{n}$, $\mathscr{E}_n$ in $[0,1]$, and an injection $\w: \bigcup_{k=0}^n\mathscr{E}_k \to \N^*$ with properties (P1)--(P7). The new collections of intervals will be
\[ \mathscr{B}_{n+1} = \mathscr{B}_n \cup \bigcup_{J \in \mathscr{E}_n}\mathscr{B}_{n+1}(J), \qquad \mathscr{E}_{n+1} =  \bigcup_{J \in \mathscr{E}_n}\mathscr{E}_{n+1}(J).\]
If $S\in \mathscr{B}_n$, then we set $f_{n+1}|S = f_n|S$.

Fix now $J\in \mathscr{E}_n$ and write $J = [t,s]$. By (P3) there exist distinct $i,j \in\N$ such that $|\w(J)|\geq n$, $f_n(t) \in \phi_{\w(J)i}(K)$, and $f_n(s) \in \phi_{\w(J)j}(K)$. There also exists an interval $I=[a,b] \subset [0,1]$ such that $\gamma(a) \in \phi_i(K)$ and $\gamma(b) \in \phi_j(K)$. We consider three possible cases.

\emph{Case I:} Suppose that $\phi_{\w(J)}\circ\gamma(a) = f_n(t)$ and $\phi_{\w(J)}\circ\gamma(b) = f_n(s)$. Set $\mathscr{B}_{n+1}(J) = \{J\}$,  $\mathscr{E}_{n+1}(J) = \emptyset$, and $f_{n+1}|J = \phi_{\w(J)} \circ (\gamma|I) \circ \zeta$ where $\zeta : J \to I$ is the orientation preserving linear map.

\emph{Case II:} Suppose that $\phi_{\w(J)}\circ\gamma(a) \neq f_n(t)$ and $\phi_{\w(J)}\circ\gamma(b) = f_n(s)$. Let $z \in (t,s)$ and let $u\in \N^*$ be the shortest word such that there exist distinct $i,j\in \N$ with $\phi_{\w(J)}\circ\gamma(a) \in \phi_{\w(J)ui}(K)$ and $f_n(t) \in \phi_{\w(J)uj}(K)$. By (P4) we have that $|u|\geq 1$. Set $\mathscr{E}_{n+1}(J) = \{[t,z]\}$,  $\mathscr{B}_{n+1}(J) = \{[z,s]\}$, and $\w([t,z]) = \w(J)u$. Define $f_{n+1}$ on $J$ continuously so that $f_{n+1}|[z,s]$ is as in Case I, and $f_{n+1}|[t,z]$ is linear with $f_{n+1}(t) = f_n(t)$. We work similarly if $\phi_{\w(J)}\circ\gamma(a) = f_n(t)$ and $\phi_{\w(J)}\circ\gamma(b) \neq f_n(s)$.

\emph{Case III:} Suppose that $\phi_{\w(J)}\circ\gamma(a) \neq f_n(t)$ and $\phi_{\w(J)}\circ\gamma(b) \neq f_n(s)$. Let $t<z<w<s$ and set $\mathscr{E}_{n+1}(J) = \{[t,z], [w,s]\}$,  $\mathscr{B}_{n+1}(J) = \{[z,w]\}$. For the definitions of $\w([t,z])$, $\w([w,s])$, and $f_{n+1}|J$ we work as in Cases I, II.

Properties (P1)--(P7) are easy to verify. Since $\w$ is injective on $\bigcup_{k=0}^n\mathscr{E}_k$, then by (P3) and the fact that $\w$ is injective on $\mathscr{E}_{n+1}(J)$ for each $J\in \mathscr{E}_n$ we have that $\w$ is injective on $\bigcup_{k=0}^{n+1}\mathscr{E}_k$. Finally, continuity of $f_{n+1}$ follows from the facts that $f_{n+1}$ is the same as $f_n$ outside of intervals in $\mathscr{E}_n$, that $f_{n+1}|J$ is continuous for all $J\in \mathscr{E}_n$, and that $f_{n+1}|\partial J = f_{n}|\partial J$ for all $J\in \mathscr{E}_n$. This completes the inductive construction.

From (P5) and (P7) we have that $(f_n)_{n\in\N}$ converges to a continuous map $f:[0,1] \to \ell_{\infty}$. Fix $J\in \bigcup_{n\geq 0}\mathscr{E}_n$. By (P5) we have that for all $S\in \bigcup_{n\geq 0}\mathscr{B}_n$ with $S\subset J$, $f(S)\subset \phi_{\w(J)}(K)$. By this fact, by (P3), and by (P4) we have that for all $m\geq n$,
\[ \sup_{t\in J}\dist(f_m(t),\phi_{\w(J)}(K)) \leq \max_{\substack{J' \subset J \\ J' \in \mathscr{E}_m}} 2L_{\w(J')} \leq 2L_{\w(J)}L_1^{m-n}.\]
Since $K$ is closed, it follows that $f(J) \subset \phi_{\w(J)}(K)$. This proves the first part of the lemma.

Assume now that $\gamma$ is $\frac{1}{s}$-H\"older for some $s>\sdim(\mathcal{F})$ and with H\"older constant $H_0$. Then, $\psi_{\mathcal{F}}(s) \in (0,1)$. Set $\mathscr{E} = \bigcup_{n\geq 0}\mathscr{E}_n$ and $\mathscr{B} = \bigcup_{n\geq 0}\mathscr{B}_n$. For each $J \in \mathscr{E}$, set 
\[ \mathcal{M}(J) := \sum_{\substack{J'\in \mathscr{E}\\ J'\subset J}}(L_{\w(J')})^{s}.\]
By injectivity of $\w$ and (P3) we have that for each $J \in \mathscr{E}$
\begin{align*}
 (L_{\w(J)})^s \leq \mathcal{M}(J) \leq \sum_{w\in \N^*}(L_{\w(J)w})^{s}  = (L_{\w(J)})^s \sum_{n=0}^{\infty} \left( \sum_{i\in\N}L_i^{s} \right)^n = \frac{ (L_{\w(J)})^s}{1-\psi_{\mathcal{F}}(s)}.
\end{align*}

The only difference in the construction in this case is that we require that for all $S\in \mathscr{B}$ and all $J\in \mathscr{E}$
\[ |S| = (L_{\w(J_S)})^{s}/\mathcal{M}([0,1])\qquad \text{and}\qquad |J| = \mathcal{M}(J)/\mathcal{M}([0,1]).\]
To see why this is possible, note that if $J\in \mathscr{E}_n$ and $\mathscr{B}_{n+1}(J)=\{S\}$ for some $n\geq 0$, then
\[ |J| = \frac{(L_{\w(J)})^{s}}{\mathcal{M}([0,1])} + \sum_{\substack{J'\in \mathscr{E}\\ J'\subsetneq J}}\frac{(L_{\w(J')})^{s}}{\mathcal{M}([0,1])} = |S| + \sum_{J' \in \mathscr{E}_{n+1}(J)}|J'|.\]

We claim that the map $f$ defined above is $\frac{1}{s}$-H\"older continuous with H\"older constant depending only on $\psi_{\mathcal{F}}(s)$, $L_1$, $s$, and $H_0$. To show the claim, fix $p,q\in [0,1]$. Clearly, we may assume that $p\neq q$. There exist $n\in\N$ and $J\in \mathscr{E}_n$ so that $p,q\in J$ and $n$ is maximal. Proving the claim falls to a case study. 

\emph{Case 1.} Suppose that $p\in J_{1}$ and $q\in J_{2}$ where $J_1,J_2 \in \mathscr{E}_{n+1}(J)$ are distinct. On the one hand 
\[ |p-q|\geq |S| \gtrsim_{s,\psi_{\mathcal{F}}(s)} (L_{\w(J)})^s \]
while on the other hand, by (P5) and (P7), $f(p),f(q) \in \phi_{\w(J)}(K)$, so
\begin{align*} 
d(f(p),f(q)) \leq L_{\w(J)}.
\end{align*}

\emph{Case 2.} Suppose that $p,q\in S$ where $S \in \mathscr{B}_{n+1}(J)$. By (P5), 
\[ d(f(p),f(q))\leq  L_{\w(J)} |S|^{-1/s} H_0 |p-q|^{1/s} \lesssim_{s,\psi_{\mathcal{F}}(s),H_0}|p-q|^{1/s}.\]

\emph{Case 3.} Suppose that $p\in J_{1}\setminus S$ and $q\in S\setminus J_1$ where $J_1 \in \mathscr{E}_{n+1}(J)$ and $S \in \mathscr{B}_{n+1}(J)$. There exists integer $m\geq n+1$ and $J' \in \mathscr{E}_{m}$ such that $p\in J' \subset J_1$, $J' \cap S \neq \emptyset$, and if $S' \in \mathscr{B}_{m+1}(J')$, then $S'$ separates $p$ from $q$. Let $z$ be the unique point in $J' \cap S$. By Case 1 for $p,z$ and Case 2 for $z,q$
\begin{align*} 
d(f(p),f(q)) &\leq d(f(p),f(z)) + d(f(z),f(q)) \lesssim_{L_1,s,\psi_{\mathcal{F}}(s),H_0} |p-q|^{1/s}
\end{align*}
which completes the proof of the claim.
\end{proof}

\subsection{Parameterizations of IIFS attractors}

The second step in the proof of Theorem \ref{thm:3} is the following lemma that allows us to reparametrize $\gamma$ so that preimages of $\phi_i(K)$ have nonempty interior.

\begin{lemma}\label{lem:reparam}
Let $p$ be in the image of $\g$. There exists a map $\G: [0,1] \to K$ and a collection of nondegenerate closed intervals $\{\mathcal{I}_n\}_{n\in\N}$ with disjoint interiors such that $\G$ has the same image as $\g$, satisfies $\G(0)=\G(1)=p$, and  for each $n\in\N$, $\mathcal{I}_n \subset \G^{-1}(\phi_n(K))$. Moreover, if $\g$ is $\frac{1}{s}$-H\"older with constant $H_0$, and if $(a_n)\in \ell^1$ is a sequence of positive numbers, then $\G$ is $\frac{1}{s}$-H\"older with constant $H\leq 2^{1/s}H_0(1+\|(a_n)\|_1)^{1/s}$ and for each $n\in\N$, $|\mathcal{I}_n| =a_n(1 + \|(a_n)\|_1)^{-1}$.
\end{lemma}

\begin{proof}
We may assume that $\g(0)=\g(1)=p$. For each $n\in\N$ fix a point $x_n \in \gamma^{-1}(\phi_n(K))\subset [0,1]$. It is possible that for some $n\neq m$ we have $x_n=x_m$. Let $\{p_k\}_{k\in B}$ be an enumeration of the set $\{x_n\}_{n\in\N}$ where $B$ is either a finite set, or $\N$. For each $k\in B$, define $A_k = \{n\in\N : x_n =p_k\}$. 

For the first claim, fix a decreasing sequence $(b_n)_{n\in\N}$ of positive numbers that converges to 0. Identify $\R^3$ with $\mathbb{C}\times \R$ and define the set
\[ E = \left(\{0\}\times[0,1]\right) \cup \bigcup_{k\in B}\bigcup_{n\in A_k}\left( \{t e^{2\pi i/n} : t\in [0,b_n]\}\times\{p_k\}\right) \subset \R^3.\]
Since $b_n\to 0$, it is easy to see that $E$ is closed. Moreover, there exists a continuous increasing map $\eta: [0,1] \to [0,1]$ such that $\eta(0)=0$,
\[ \max_{n\in A_{k}}b_n \leq \eta\left( \min_{j\in\{1,\dots,k-1\}}|p_j-p_k|  \right)\qquad\text{for all $k\in B$}\] 
and
\[ b_n \leq \eta\left( \min_{m\in\{1,\dots,n-1\}} |e^{2\pi i/m}-e^{2\pi i/n}|  \right)\qquad\text{for all $n\in \N$}.\] 

We claim that there exists a continuous $\omega :[0,1]\to [0,1]$ with $\omega(0)=0$ such that for all $x,y \in E$, there exists a curve $\sigma:[0,1]\to E$ such that $\sigma(0)=x$, $\sigma(1)=y$ and the diameter of its image is at most $\omega(|x-y|)$. The proof of the claim is a simple case study.

If $x,y \in \{0\}\times[0,1]$, then use the line segment $[x,y]$. 

If $x = (t e^{2\pi i/n},p_k)$ and $y\in \{0\}\times [0,1]$ for some $t \in [0,b_n]$, $n\in A_k$ and $k\in B$, then use the union of line segments $[x,(0,p_k)] \cup [(0,p_k),y]$. 

If $x=(t_1 e^{2\pi i/n},p_k)$ and $y=(t_2 e^{2\pi i/n},p_k)$ for some distinct $t_1,t_2 \in [0,b_n]$, $n\in A_k$ and $k\in B$, then use the line segment $[x,y]$. 

If $x=(t_1 e^{2\pi i/n},p_k)$ and $y=(t_2 e^{2\pi i/m},p_k)$ for some $t_1,t_2 \in [0,b_n]$, $n,m\in A_k$ with $m<n$, and $k\in B$, then use the union of line segments $[x,(0,p_k)] \cup [(0,p_k),y]$. Note that
\begin{align*} 
|x-y| &\gtrsim  |t_1-t_2| + \min\{t_1,t_2\} |e^{2\pi i/m}-e^{2\pi i/n}| \\
&\geq |t_1-t_2| + \min\{t_1,t_2\} \eta^{-1}(\min\{t_1,t_2\})
\end{align*}
while
\begin{align*} 
\diam(\sigma([0,1])) &\leq t_1+t_2 \leq |t_1-t_2| + \min\{t_1,t_2\}.
\end{align*}

Finally, suppose that $x = (t_1 e^{2\pi i/m},p_k)$ and $y = (t_2 e^{2\pi i/n},p_j)$, for some distinct $k,j \in B$, $m\in A_k$, $n\in A_j$, $t_1 \in [0,b_m]$, and $t_2 \in [0,b_n]$. Then, 
\[ |x-y| \gtrsim |p_k-p_j| + |t_1-t_2|\] 
and if $\sigma$ is the union of line segments $[x,(0,p_k)]\cup [(0,p_k),(0,p_j)] \cup [(0,p_j),y]$,
\begin{align*} 
\diam(\sigma([0,1])) \leq t_1+t_2 + |p_k-p_j| &\leq |t_1-t_2| + \min\{t_1,t_2\} + |p_k-p_j|\\
&\lesssim |x-y| + \eta(|p_k-p_j|)
\end{align*}
and the proof of the claim is complete.

Thus, $E$ is connected and locally connected, and by the Hahn-Mazurkiewicz Theorem \cite[Theorem 3.30]{HY}, there exists continuous surjection $g:[0,1] \to E$. Note that for each $n\in\N$, the preimage
\[ g^{-1}(\{t e^{2\pi i/n} : t\in [0,b_n]\}\times\{x_n\})\]
contains a nondegenerate closed interval. Define now $\tilde{\g} : E \to K$ by $\tilde{\g}|\{0\}\times[0,1] = \g$ and for each $n\in\N$ and $t\in[0,b_n]$, $\tilde{\g}(t e^{2\pi i/n},x_n) = \g(x_n)$. Then $\tilde{\g}$ is continuous and $\G:= \tilde{\g}\circ g$ satisfies the conclusions of the lemma.

For the second part of the lemma, assume that $\g$ is $\frac1s$-H\"older with constant $H_0$, and assume that $(a_n)\in \ell^1$ is a sequence of positive numbers. Define $E$ as above replacing $b_n$ by $a_n$. Then, 
\[ \mathcal{H}^1(E) = 1 + \|(a_n)\|_1 <\infty.\]
Therefore, by \cite[Theorem 4.4]{AO}, there exists a Lipschitz surjection $g:[0,1] \to E$ with constant speed equal to $2\mathcal{H}^1(E)$. Thus, for each $n\in\N$, $g^{-1}(\{t e^{2\pi i/n} : t\in [0,a_n]\}\times\{x_n\})$ contains a closed subinterval of length
\[ (1+\|(a_n)\|_1)^{-1}a_n.\]
Define $\tilde{\g}$ as above and note that $\tilde{\g}$ is $\frac1{s}$-H\"older with constant $H_0$. Setting $\G := \tilde{\g}\circ g$, we have for all $x,y\in [0,1]$
\[ d(\G(x),\G(y)) \leq H_0 |g(x) - g(y)|^{1/s} \leq 2^{1/s}H_0(1+ \|(a_n)\|_1)^{1/s} |x-y|^{1/s}. \qedhere\]
\end{proof}

We are now ready to prove Theorem \ref{thm:3}.

\begin{proof}[Proof of Theorem \ref{thm:3}]
Fix $p_0 \in K$. We construct a sequence of continuous maps $(f_n:[0,1]\to K)_{n\geq 0}$, sequences of collections of closed nondegenerate intervals $(\mathscr{N}_n)_{n\geq 0}$, $(\mathscr{I}_n)_{n\geq 0}$ in $[0,1]$, and a bijection $\w:\bigcup_{n\geq 0}\mathscr{N}_n \to \N^*$ with the following properties.
\begin{enumerate}
\item[{(P1)}] For each $n\geq 0$, intervals in $\mathscr{N}_{n}\cup \mathscr{I}_n$ intersect at most on endpoints. 
\item[{(P2)}] For each $n\geq 0$, each $I \in \mathscr{N}_n$ and each $i\in \N$, there exists $I'\in\mathscr{N}_{n+1}$ contained in the interior of $I$ such that $\w(I') = \w(I)i$. Conversely, for each $n\geq 0$ and each $I' \in \mathscr{N}_{n+1}$, there exists $I\in\mathscr{N}_{n}$ and $i\in\N$ such that $I'$ is contained in the interior of $I$ and $\w(I') = \w(I)i$.
\item[{(P3)}] For each $n\geq 0$ and each $I \in \mathscr{N}_n$, there exist exactly two intervals $J,J' \in \mathscr{I}_{n+1}$ contained in $I$. Conversely, for each $n\geq 0$ and each $J\in \mathscr{I}_{n+1}\setminus \mathscr{I}_{n}$, there exists unique interval $I_J \in \mathscr{N}_n$ such that $J \subset I_J$.
\item[{(P4)}] If $n\geq 1$, $I\in\mathscr{N}_{n-1}$ and $J,J' \in \mathscr{I}_n$ are contained in $I$, then there exists an orientation reversing linear map $\zeta_{J}:J' \to J$, such that $f_n|J' = (f_n|J)\circ\zeta_{J}$. Moreover, $f_{n}|J = g\circ \zeta$ where $g:[0,1] \to \phi_{\w(I)}(K)$ is the map from Lemma \ref{lem:conn} and $\zeta : J \to [0,1]$ is an increasing linear map.  
\item[{(P5)}] For each $n\geq0$ and each $I \in \mathscr{N}_n$, $f_n|I$ is constant and its image is in $\phi_{\w(I)}(K)$. Moreover, $f_n(I) \subset f_{n+1}(I) \subset \phi_{\w(I)}(K)$.
\item[{(P6)}] For each $n\in\N$, if $x$ is not in the interior of some $I\in \mathscr{N}_n$, then $f_{n+1}(x)=f_n(x)$.
\item[{(P7)}]  For each $n\in\N$, $f_n(0)=f_n(1)=p_0$.
\end{enumerate}

For the construction, apply Lemma \ref{lem:reparam} and obtain a map $\G:[0,1] \to K$ and a collection $\{\mathcal{I}_i : i\in\N\}$ of closed nondegenerate intervals in $[0,1]$ such that for each $i\in\N$, $\Gamma|\mathcal{I}_i$ is constant and its image is in $\phi_i(K)$. The proof of the claim is done in an inductive fashion.

For $n=0$, let $\mathscr{N}_0 = \{[0,1]\}$ and $\mathscr{I}_0 = \emptyset$, let $f_0 : [0,1] \to K$ be the constant map $p_0$, and let $\w([0,1]) = \varepsilon$. Properties (P1), (P5) and (P7) are trivial while the rest of them are vacuous.

Assume now that for some $n\geq 0$ we have defined a continuous $f_n : [0,1] \to K$, collections of intervals $\mathscr{N}_n,\mathscr{I}_n$, and a bijection $\w:\mathscr{N}_n \to \N^n$ satisfying assumptions (P1)--(P7). The new collections of intervals will be
\[ \mathscr{I}_{n+1} = \mathscr{I}_n \cup \bigcup_{I \in \mathscr{N}_n}\mathscr{I}_{n+1}(I), \qquad \mathscr{N}_{n+1} =  \bigcup_{I \in \mathscr{N}_n}\mathscr{N}_{n+1}(I).\]
We set 
\[ f_{n+1}| [0,1]\setminus \bigcup\mathscr{N}_n = f_n | [0,1]\setminus \bigcup\mathscr{N}_n.\]

Fix now $I\in \mathscr{N}_n$. By (P5), there exists $i\in\N$ such that $f_n(I) \in \phi_{\w(I)i}(K)$. Reparameterizing $\G$, we may assume that $\G(0)\in \phi_j(K)$ and $0$ is the left endpoint of $\mathcal{I}_j$. Write $I=[a,b]$, and let $a < a_1 < a_2 < b$. Let $\xi_I : [a_1,a_2] \to [0,1]$ be an increasing linear map and let $\zeta_I : [a,a_1] \to [0,1]$ and $\zeta_I' : [a_2,b] \to [0,1]$ be increasing linear maps. Set
\[ \mathscr{N}_{n+1}(I) = \{ \xi_I^{-1}(\mathcal{I}_{j}) : j\in\N\} \qquad\text{and}\qquad \mathscr{I}_{n+1}(I) = \{[a,a_1], [a_2,b]\}\]
and for each $j\in\N$ define $\w(\xi_I^{-1}(\mathcal{I}_{j})) = \w(I)j$.

Let $g: [0,1] \to \phi_{\w(I)}(K) $ be the map given from Lemma \ref{lem:conn} that connects $f_n(I)$ to $\phi_{\w(I)}(\Gamma(\mathcal{I}_j))$ and define
\begin{enumerate}
\item $f_{n+1}|[a,a_1] = g\circ \zeta_I$,
\item $f_{n+1}|[a_1,a_2] = \phi_{\w(I)}\circ \Gamma\circ\xi_{I}$,
\item $f_{n+1}|[a_2,b] = g\circ h \circ \zeta_I'$ where $h:[0,1] \to [0,1]$ with $h(x) = 1-x$.
\end{enumerate}

Properties (P1)--(P6) are clear from design and the properties of $\Gamma$. Note that for all $I \in \mathscr{N}_n$ the function $\w : \mathscr{N}_{n+1}(I) \to \{\w(I)i : i\in\N\}$ is bijective. Therefore, $\w:\mathscr{N}_{n+1} \to \N^{n+1}$ is a bijection. Finally, since $0,1$ are not contained in the interior of any $I\in\mathscr{N}_n$, by (P6), $f_{n+1}(0)=f_{n+1}(1)=p_0$.

It remains to prove continuity of $f_{n+1}$. 
Fix $x\in[0,1]$. We only show continuity of $f_{n+1}$ at $x$ from the right. To this end, fix a sequence $x_m \subset (x,1]$ that converges to $x$ and consider the following three cases.

\emph{Case 1.} Suppose that for all $m$ sufficiently large, $x_m \in [0,1]\setminus\overline{\bigcup \mathscr{N}_n}$. Then $f_{n+1}(x_m)$ converges to $f_{n+1}(x)$ by (P6) and continuity of $f_{n}$.

\emph{Case 2.} Suppose that for all $m$ sufficiently large, $x_m \in I$ for some $I\in\mathscr{N}_n$. Then, $f_{n+1}(x_m)$ converges to $f_{n+1}(x)$ by design.

\emph{Case 3.} Suppose that for for all $m$ sufficiently large, there exists $I_m \in \mathscr{N}_n$ such that $x_m \in I_m$ and that the collection $\{I_m\}_m$ is infinite. Fixing $\e>0$, there exists $i_0\in\mathbb{N}$ such that for every $i\geq i_0$, $L_i<\e/2$. Since the collection $\{\w(I_m)\}_m$ is infinite, there exists $N\in\N$ such that for every $m\geq N$, some character of the word $\w(I_m)$ is larger than $i_0$. It follows that $L_{\w(I_m)}< \e/2$ for every $m\geq N$. By continuity of $f_n$, we may further assume that for every $m\geq N$, $d(f_n(x),f_n(x_m))<\e/2$. By (P5), for every $m\geq N$,
\begin{align*}
d(f_{n+1}(x),f_{n+1}(x_m))\leq d(f_n(x),f_n(x_m))+d(f_n(x_m),f_{n+1}(x_m)) <\e.
\end{align*}

This completes the induction and the proof of (P1)--(P7). 

By (P5) and (P6),
\[ \|f_{n+1} - f_n\|_{\infty} \leq \sup_{I\in \mathscr{N}_n}\|f_n - f_{n+1}\|_{I,\infty} \leq \sup_{I\in \mathscr{N}_n} L_{\w(I)} \leq L_1^n\]
so the maps $f_n$ converge uniformly to a continuous map $f:[0,1] \to K$. By (P5) and the bijectivity of $\w$, we have $f([0,1])\cap \phi_w(K) \neq \emptyset$ for all $w\in \N^*$. Therefore, for all $x\in K$ and $n\in\N$,
\[ \dist(x,f([0,1])) \leq  \inf_{w\in\N^n} \diam{\phi_w(K)} \leq L_1^n.\]
Hence, $K\subset f([0,1])$ and it follows that $f([0,1]) = K$. This proves the first part of Theorem \ref{thm:3}.

For the second part of the theorem, assume that $\gamma$ is $\frac{1}{s}$-H\"older for some $s>\sdim(\mathcal{F})$ and with H\"older constant $H_0$. Define for all $w\in\N^*$
\[ M_w := 3\sum_{u\in\N^*}(L_{wu})^s.\]
Working as in the proof of Lemma \ref{lem:conn}, we have that for all $w \in \N^*$,
\[  M_w = 3\left(L_{w}\right)^s(1-\psi_{\mathcal{F}}(s))^{-1}.\]
For each $i\in\N$ set $a_i = M_i >0$ and note that $\|(a_i)\|_1=M_{\varepsilon} \lesssim_{s,\psi_{\mathcal{F}(s)}} 1$. We apply on each stage of the construction, the second part of Lemma \ref{lem:reparam} with $a_i = M_i$ and we may assume that $\G$ is $\frac{1}{s}$-H\"older with constant $H\lesssim_{s,\psi_\mathcal{F}(s),L_1} H_0$. 

The other change in the construction, is that we require that if $n\geq 0$, $I\in\mathscr{N}_n$, and $J\in \mathscr{J}_{n+1}(I)$, then
\[ |I| = M_{\varepsilon}^{-1}M_{\w(I)} \quad\text{and}\quad  |J| = M_{\varepsilon}^{-1} \left(L_{\w(I)}\right)^s.\]
To see why this is possible, fix $n\geq 0$ and $I\in \mathscr{N}_n$. By Lemma \ref{lem:reparam}
\begin{align*} 
|I| &= M_{\varepsilon}^{-1}M_{\w(I)}\\ 
&= 2M_{\varepsilon}^{-1} \left(L_{\w(I)}\right)^s + M_{\varepsilon}^{-1} \left(L_{\w(I)}\right)^s + \sum_{i\in\N}M_{\varepsilon}^{-1}M_{\textbf{w}(I)i}\\
&=\sum_{J\in \mathscr{J}_{n+1}(I)}|J| + \left|I\setminus \bigcup\mathscr{J}_{n+1}(I) \setminus \bigcup\mathscr{N}_{n+1}(I)\right| + \sum_{J\in \mathscr{N}_{n+1}(I)}|J|.
\end{align*}

We claim that the limit $f$ of the maps $f_n$ is $\frac{1}{s}$-H\"older continuous. To this end, fix distinct $p,q\in [0,1]$ and let $n\geq 0$ be the maximal integer such that there is some $I\in\mathscr{N}_n$ with $p,q\in I$. Denote by $J,J'$ the two elements in $\mathscr{J}_{n+1}(I)$, by $I'$ the closure of $I \setminus (J\cup J')$, and by $\mathcal{B}_I$ the closure of the set of points in $I$ which are not contained in any interval in $\mathscr{N}_{n+1}(I)\cup \mathscr{J}_{n+1}(I)$. The proof of the claim falls to a case study.

\emph{Case 1.} Suppose that $p,q\in J$ or $p,q\in J'$. By (P6) and Lemma \ref{lem:conn}, 
\begin{align*}
d(f(p),f(q)) &= d(g\circ\zeta_I(p),g\circ\zeta_I(q))\\
&\lesssim_{\psi_\mathcal{F}(s),L_1,s,H_0} L_{\w(I)}|J|^{-1/s}|p-q|^{1/s}\\
&\lesssim_{\psi_\mathcal{F}(s),s,L_1,H_0}|p-q|^{1/s}.
\end{align*}

\emph{Case 2.} Suppose that $p$ and $q$ are separated by the interior of one of $I'$, $J$, $J'$. Then $|p-q| \gtrsim_{s,\psi_{\mathcal{F}}(s)} (L_{\w(I)})^s$ while, by (P5), $f(I) \subset \phi_{\w(I)}(K)$ and 
\[d(f(p),f(q))\leq L_{\w(I)}.\]

\emph{Case 3.} Suppose that $p,q \in \mathcal{B}_I$. By (P6) and design of $f_{n+1}$, $f|\mathcal{B}_I = f_{n+1}|\mathcal{B}_I = \phi_{\w(I)}\circ\G\circ\xi_I$. Therefore, by Lemma \ref{lem:reparam} we have
\[d(f(p),f(q))\lesssim_{H_0,s,\psi_\mathcal{F}(s)} L_{\w(I)} |I'|^{-1/s} |p-q|^{1/s}\lesssim_{H_0,\psi_\mathcal{F}(s),s}|p-q|^{1/s}.\]

\emph{Case 4.} Suppose that $p\in I_1$ and $q\in I_2$ where $I_1,I_2\in\mathscr{N}_{n+1}(I)$. By maximality of $n$, we have $I_1\neq I_2$. Let $a\in I_1$ and $b\in I_2$ such that $|a-b|=\dist(I_1,I_2)$. Then the pair $p,a$ satisfy either Case 1 or Case 2, with $I$ replaced by $I_1$. Similarly for $q,b$. Moreover, $a,b \in \mathcal{B}_I$, and hence satisfy Case 3. Therefore, by triangle inequality, 
\[d(f(p),f(q)) \lesssim_{\psi_\mathcal{F}(s),s,H_0,L_1} |p-q|^{1/s}.\]

\emph{Case 5.} Suppose that $p\in I_1$ for some $I_1\in\mathscr{N}_{n+1}(I)$ and $q\in\mathcal{B}_I$. Let $a\in I_1$ be such that $|a-q|=\dist(I_1,q)$. Note that points $a,p$ satisfy one of Case 1 or Case 2 (with $I$ replaced by $I_1$), while points $a,q$ satisfy Case 3. Therefore, by the triangle inequality,
\[d(f(p),f(q))\lesssim_{\psi_\mathcal{F}(s),s,H_0,L_1} |p-q|^{1/s}.\]

\emph{Case 6.} Suppose that $p\in J\cup J'$ (say $J$) and $q\in I'$. Let $a\in J$ such that $|a-q|=\dist(J,q)$. Note that points $a,p$ satisfy Case 1, while points $a,q$ satisfy Case 3 or Case 5. Therefore, by the triangle inequality,
\[d(f(p),f(q))\lesssim_{\psi_\mathcal{F}(s),s,H_0,L_1} |p-q|^{1/s}. \qedhere\]
\end{proof}

\section{Examples of IIFS}\label{sec:ex}

In this section we provide three examples of IIFS. In \textsection\ref{sec:ex3} we show that the condition $\lim_{n\to \infty} \Lip(\phi_n) =0$ is necessary in Theorem \ref{thm:3}, in \textsection\ref{sec:ex1} we prove Theorem \ref{thm:1}(1), and in \textsection\ref{sec:ex2} we prove Theorem \ref{thm:1}(2).

\subsection{An IIFS without vanishing Lipschitz norms}\label{sec:ex3}
For this example we use complex coordinates. For each $n\in\N$ define a contraction $\phi_n: \overline{\mathbb{B}^2} \to \overline{\mathbb{B}^2}$ on the closed unit disk $\overline{\mathbb{B}^2}$ by
\[ \phi_n(z) = e^{\frac{2\pi i}{n}}\tfrac12(\text{Re}(z) + 1).\]
Note that $\Lip(\phi_n) = 1/2$ for all $n\in\N$, and let $\mathcal{F}=\{\phi_n:n\in\N\}$.

Set $G = \{te^{\frac{2\pi i}{n}} : n\in\N, t\in [0,1]\}$; see Figure \ref{fig3} below for the first generation of images of the system used in this section. Note that the figure shows only the images of finitely many maps in the family, the images in fact accumulate to the real interval $[0,1]$.

\begin{figure}[h]
 \centering
	\begin{minipage}{.49\textwidth}
	\centering
        \includegraphics[width=0.95\textwidth]{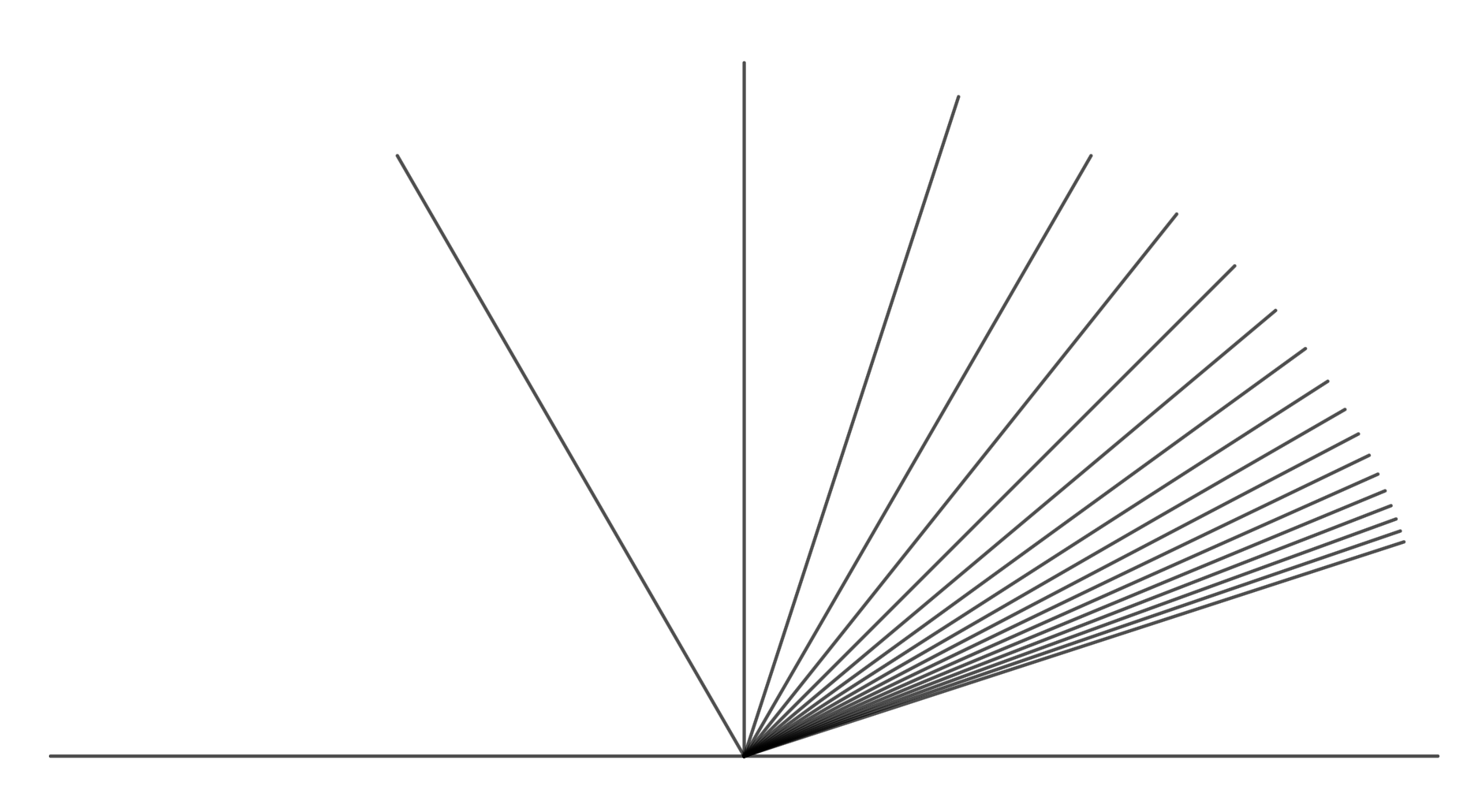}
	\end{minipage}\hfill
        \caption{Attractor of IIFS of \textsection \ref{sec:ex3}}\label{fig3}
\end{figure}

We claim that $G$ is the attractor $K$ of the IIFS $\mathcal{F}$. Note that for all $n\in\N$ we have $\phi_n(\overline{\mathbb{B}^2}) = \phi_n(G)$. Therefore, $\phi_w(\overline{\mathbb{B}^2}) = \phi_w(G)$ for all $w\in \N^*$. Moreover, it is easy to see that
\begin{equation*}\label{eq:book}
 \bigcup_{n\in\N}\phi_n(\overline{\mathbb{B}^2}) = G,
\end{equation*} 
which yields that $K\subset G$. For the opposite inclusion, fix $x \in G$. There exists $n_1 \in \N$ such that $x\in \phi_{n_1}(\overline{\mathbb{B}^2})$. Assume now that for some $m\in\N$ we have defined a word $w \in \N^m$ such that $x \in \phi_w(\overline{\mathbb{B}^2}) = \phi_n(G)$. Since $G = \bigcup_{n\in\N}\phi_n(\overline{\mathbb{B}^2})$, there exists $n_{m+1} \in \N$ such that $x \in \phi_{w n_{m+1}}(\overline{\mathbb{B}^2})$. It follows that there exists an infinite word $w=n_1n_2\cdots \in \N^{\N}$ such that $x \in \bigcap_{m\in\N}\phi_{n_1\cdots n_m}(\overline{\mathbb{B}^2})$ which yields that $G\subset K$. 

To complete the example, note that there exists a curve $\gamma:[0,1] \to K$ (namely the constant curve with image the origin) whose image intersects every $\phi_n(K)$, and the attractor $K$ is a continuum but it is not locally connected.

\subsection{An IIFS where the attractor is a continuum but not path connected}\label{sec:ex1}

Fix $s>1$ and fix $M\in\N$ such that $M > \max \left\{ 4^{\frac1{s-1}}, 7 \right\}$. For each $n\in\N$ let
\[ a_n = \frac{1+ M^{n}(2M+1)^{1-n}}{M+1} \quad\text{and}\quad b_n = M^{-n}(a_n-a_{n+1}-M^{-n}).\]
Note that $a_n \in (0,1]$ for each $n$, that $a_n$ is strictly decreasing, that $a_1 =1$, and that for all $n\in\N$
\[ a_{n} - a_{n+1} - M^{-n} = \left( \frac{M}{2M+1} \right)^{n}-M^{-n} > M^{-n}. \] 
The construction of the IIFS $\mathcal{F}$ is done in three subfamilies
\[ \mathcal{F} = \{\phi_{n,i} \}_{n,i} \cup \{\tau_{n,i}\}_{n,i} \cup\{\sigma_n \}_n.\]

First, for $n\in\N$ and $i\in \{1,\dots,M^n\}$, define $\phi_{n,i}: [0,1]^2 \to [0,1]^2$ by 
\[\phi_{n,i}((x,y)) = (a_{n} - M^{-n},M^{-n}(i-1)) + M^{-n}(x,y).\]

Second, for $n\in\N$ and $i\in \{1,\dots,M^n\}$, define $\tau_{n,i} : [0,1]^2 \to [0,1]^{2}$ by 
\begin{align*} 
\tau_{n,i}((x,y)) &= \left(a_{n}-M^{-n} -(i-1)b_n,0\right) + b_n(x,y)\quad\text{if $n$ is even}\\
\tau_{n,i}((x,y)) &= \left(a_{n}-M^{-n} -(i-1)b_n,1-b_n\right) + b_n(x,y)\quad\text{if $n$ is odd}.
\end{align*}

Third, for $i\in\{1,\dots ,M+1\}$, define $\sigma_i:[0,1]^2 \to [0,1]^2$ by
\[ \sigma_i((x,y)) = \left( 0, \tfrac{i-1}{M+1}\right) +   \tfrac{1}{M+1}(x,y).\]

See Figure \ref{fig5} below for the first generation of images of the system used in this section. Note that this figure shows only finitely many of the first generation images, and the images in fact accumulate to the boxes along the left edge of the square.

\begin{figure}[h]
        \centering
        \includegraphics[width=0.6\textwidth]{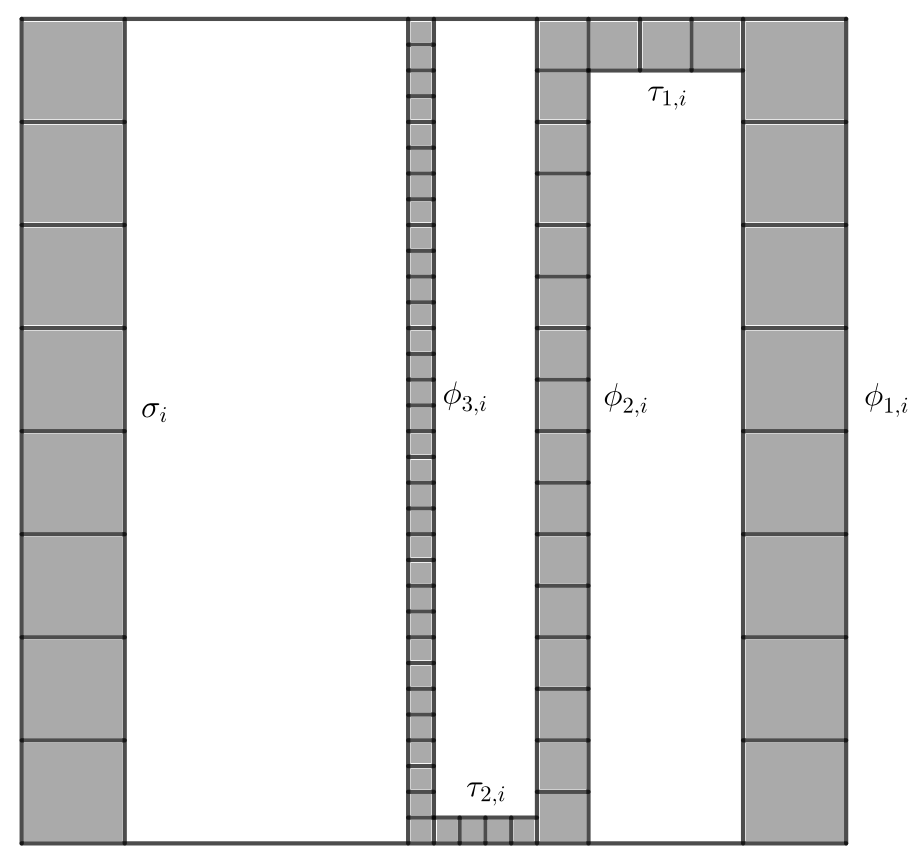}
	\caption{First iteration of IIFS of \textsection\ref{sec:ex1}.}
\label{fig5}
\end{figure}

Using the definition of $M$ we now compute
\begin{align*} 
\psi_{\mathcal{F}}(s) &= \sum_{n=1}^{\infty}\sum_{i=1}^{M^n} \Lip(\phi_{n,i})^s + \sum_{n=1}^{\infty}\sum_{i=1}^{M^n} \Lip(\tau_{n,i})^s + \sum_{n=1}^{M+1}\Lip(\sigma_{n})^s\\
&=\sum_{n=1}^{\infty}M^{(1-s)n} + \sum_{n=1}^{\infty}M^n b_n^{s} + (M+1)^{1-s}\\
&\leq \frac{2}{M^{s-1}-1} + \frac{1}{(M+1)^{s-1}}\\
&<1.
\end{align*}
Therefore, the similarity dimension of $\mathcal{F}$ is no more than $s$. We are now ready to prove the first part of Theorem \ref{thm:1}.

\begin{proof}[{Proof of Theorem \ref{thm:1}(1)}]
Define $K_0 = [0,1]^2$ and for each $m\in\N$ define $K_m = \bigcup_{f \in \mathcal{F}}f(K_{m-1})$.
By \eqref{eq:IIFSattractor}, $(K_m)_{m\in\N}$ is a nested family of sets with $K = \bigcap_{m\geq 0}K_m$. Thus in order to show that $K$ is a continuum, it suffices to show that each $K_m$ is a continuum. 

To show compactness, we proceed by induction. The base case $m=0$ is trivial. Suppose now that $K_m$ is compact. Let $(p_k)$ be a sequence in $K_{m+1}$ converging to some $p = (x,y)\in\R^2$. If $x < \frac{1}{2M-1}$, then for all $k$ large enough, $p_k\in\bigcup_{i}\sigma_{i}(K_{m})$ which is compact. If $x> \frac{1}{2M-1}$ then there exists a finite set $\mathcal{F}' \subset \mathcal{F}$ such that for all $k$ large enough, $p_k$ is contained in $\bigcup_{f\in \mathcal{F}'}f(K_m)$ which is compact. Finally, if $x = \frac{1}{2M-1}$, then $p\in \bigcup_{i}\sigma_{i}(K_{m})$. In either case, $p\in K_{m+1}$.

Connectedness is also shown inductively. The case $m=0$ is trivial. Suppose now that for some $m\geq 0$, the set $K_m$ is connected. Let 
\[ L=\bigcup_{i=1}^{2M-1}\sigma_i(K_m), \qquad R=K_{m+1}\setminus L = \bigcup_{n,i}\phi_{n,i}(K_{m}) \cup \bigcup_{n,i}\tau_{n,i}(K_{m}) .\] 
Note that $L$ is connected since for every $i\in \{1,\dots,2M-2\}$
\[ \sigma_{i+1}((1,0)) \in \sigma_i(K_m) \cap \sigma_{i+1}(K_m) \]
and each $\sigma_i(K_m)$ is connected by the inductive hypothesis. We also claim that $R$ is connected. This follows from the fact that sets $\tau_{n,i}(K_{m})$ and $\phi_{n,i}(K_{m})$ are connected and from the fact that for all $n\in\N$
\begin{align*}
\phi_{n,i+1}((1,0)) &\in \phi_{n,i}(K_m) \cap \phi_{n,i+1}(K_m) \\
\tau_{n,i+1}((1,0)) &\in \tau_{n,i}(K_m) \cap \tau_{n,i+1}(K_m)\\
\tau_{n,M^n}((0,0)) &\in \tau_{n,M^n}(K_m)\cap \phi_{n+1,1}(K_m) \\
\phi_{n,M^n}((0,0)) &\in \tau_{n,1}(K_m)\cap \phi_{n,M^n}(K_m).
\end{align*}

Therefore if we were to have some partition $K_{m+1}=A\cup B$ by disjoint nonempty open sets $A,B$, we must have that $A=L$ or $B=L$. However, $R$ is not a closed set, as it is disjoint from $\{\frac{1}{2M-1}\}\times[0,1]$ but its sequential closure contains $(\frac{1}{2M-1},0)$. Thus, $K_{m+1}$ is connected.

To finish the proof, we show that there is no path in $K_1$ connecting $p=(\frac{1}{2M-1},0)$ to $q=(1,0)$. Since both of these points are in $K$, and $K\subset K_1$ the latter implies that $K$ is not path connected. To this end, assume for a contradiction that $f = (f_1,f_2) :[0,1]\to K_1$ is a continuous map with $f(0)=p$, $f(1)=q$. By design of $K_1$, for any $n\in\mathbb{N}$
\begin{enumerate}
\item for any $t\in [0,1]$ with 
\[ \tfrac{1}{2M-1}+2^{-2n} < f_1(t) < \tfrac{1}{2M-1}+2^{-2n+1}-M^{-2n}\]
we have that $f_2(t) \leq (2M)^{-2n}$ and
\item for any $t\in [0,1]$ with 
\[ \tfrac{1}{2M-1}+2^{-2n-1} < f_1(t) < \tfrac{1}{2M-1}+2^{-2n}-M^{-2n-1}\]
we have that $f_2(t) \geq 1-(2M)^{-2n-1}$.
\end{enumerate}
It follows now that $f_2$ (and consequently $f$) is not continuous at $t=0$.
\end{proof}

\subsection{An IIFS where the attractor is the image of a curve but not the image of a H\"older curve}\label{sec:ex2}

Fix $s>1$ and fix an even integer $ M \geq \max\{10, 7^{\frac1{s-1}}\}$. For each $n\in\N$ let 
\[ a_n=\frac{\log(2)}{M^n\log(n+1)} \quad\text{and}\quad b_n = \frac{4 + 2M(2^{-1}-M^{-1})^n}{M+2}.\]
A simple calculation shows that for all $n\in\N$, $b_{n+1} < b_n - 2a_n < b_n$. Additionally, for each $n\in\N$ let $N_n$ be an integer such that
\[2 a_{n+1}^{-1}(b_n - a_n - b_{n+1}) \leq N_n \leq 4 a_{n+1}^{-1}(b_n - a_n - b_{n+1}).\]

The construction of the IIFS $\mathcal{F}$ is done in three subfamilies
\[ \mathcal{F} = \{\phi_{n,i}\}_{n,i} \cup\{\tau_{n,i}\}_{n,i} \cup \{\sigma\}.\]

First, for each $n \in\N$ and $i \in \{1,\dots,M^n\}$ define $\phi_{n,i}:[0,1]^2 \to[0,1]^2$ by 
\[ \phi_{n,i}((x,y)) = \left(b_n - a_n ,   (i-1)a_n\right) + a_n (x,y).\]
Second, for $n\in\mathbb{N}$ and $i \in \{1,\dots, N_n\}$ define $\tau_{n,i}:[0,1]^2 \to[0,1]^2$ by
\begin{align*} 
 \tau_{n,i}((x,y))= &\left( b_{n+1} + (i-1)\tfrac{b_n -a_n-b_{n+1}}{N_n} ,0\right)  + \tfrac{b_n -a_n-b_{n+1}}{N_n} (x,y).
\end{align*}
Third, we define $\sigma:[0,1]^2 \to [0,1]^2$ such that 
\[\sigma((x,y)) =\tfrac{4}{M+2} (x,y).\]

See Figure \ref{fig6} below for the first generation of images of the system used in this section. Note that this figure shows only finitely many of the first generation images, and the images in fact accumulate to the box in the bottom left corner of the square.

\begin{figure}[h]
        \centering
        \includegraphics[width=0.6\textwidth]{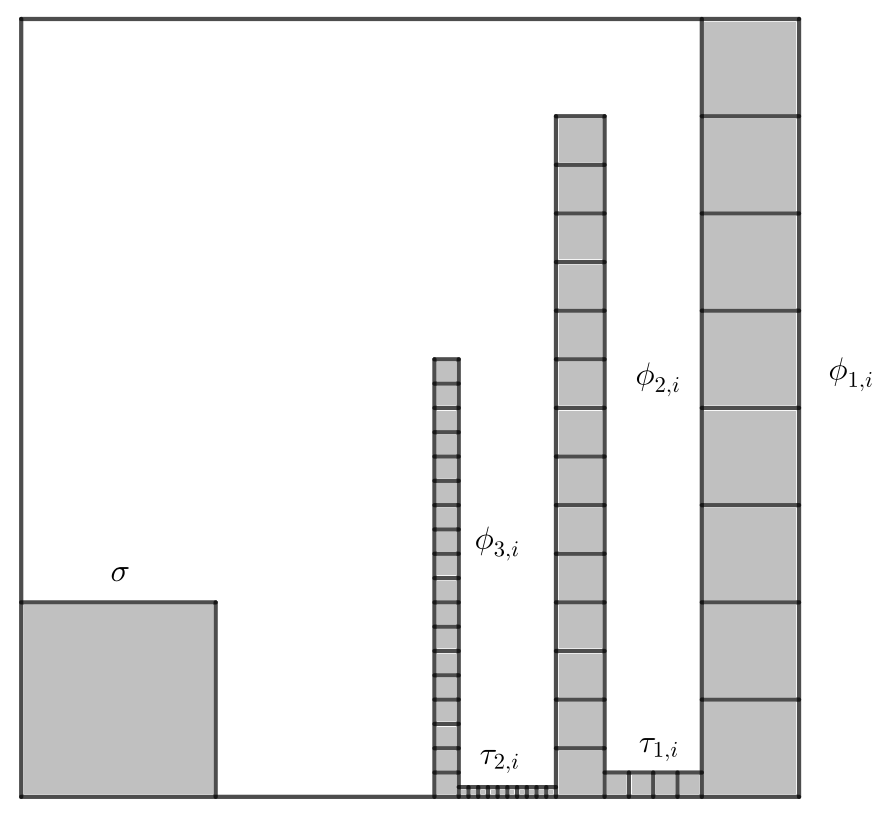} 
    \caption{First iteration of IIFS of \textsection\ref{sec:ex2}.}
\label{fig6}
\end{figure}

Using the definition of $M$ we compute 
\begin{align*}
\psi_{\mathcal{F}}(s)&=\sum_{n=1}^\infty\sum_{i=1}^{M^n} (\Lip(\phi_{n,i}))^s+\sum_{n=1}^\infty\sum_{i=1}^{N_n}(\Lip(\tau_{n,i}))^s + (\Lip(\sigma))^s.\\
&= \sum_{n=1}^\infty M^n a_n^s+\sum_{n=1}^\infty N_n\left(\frac{b_n -a_n-b_{n+1}}{N_n}\right)^s + \left(\frac{4}{M+2}\right)^s\\
&\leq \frac{M^{1-s}}{1-M^{1-s}} + \frac{M^{1-s}}{1-M^{1-s}} + \frac{4^s}{M^s}\\
&<1.
\end{align*}
Therefore, the similarity dimension of $\mathcal{F}$ is at most $s$.

\begin{lemma}\label{prop:curve}
The attractor $K$ of the IIFS $\mathcal{F}$ is the image of a curve.
\end{lemma}

\begin{proof}
Define $L_b,L_r$ to be the bottom and left, respectively, edges of the unit square $[0,1]^2$. Define also $K_0=[0,1]^2$, and for each $m\in\N$ define $K_{m} = \bigcup_{f\in\mathcal{F}}f(K_{m-1})$.

We claim that for every $f\in \mathcal{F}$, $f(L_b\cup L_r) \subset f(K)$. Assuming the claim, we note that the set 
\[ E = L_b \cup  \bigcup_{n,i}\phi_{n,i}(L_r)\]
is contained in $K$ and, working as in the proof of Lemma \ref{lem:reparam}, it follows that $E$ is the image of a curve. Now, by  Theorem \ref{thm:3}(1), it follows that $K$ is the image of a curve.

To prove the claim, we show that $L_b\cup L_r \subset K_m$ for every $m\in\mathbb{N}$. By a simple inductive argument, this immediately implies that $f(L_b\cup L_r)\subset f(K)$ for all $f\in \mathcal{F}$. The proof is by induction on $m$. Clearly $L_b\cup L_r\subset K_0$. Assume now that $L_b\cup L_r \subset K_m$ for some integer $m \geq 0$. First,
\[L_r =\bigcup_{i=1}^M\left( \{1\}\times\left[\frac{i-1}{M},\frac{i}{M}\right] \right) = \bigcup_{i=1}^{M}\phi_{1,i}(L_r) \subset K_{m+1}.\]
Second, 
\begin{align*}
L_b &= \bigcup_{n=1}^{\infty}\left([b_n-a_n,b_n] \times \{0\}\right) \cup \bigcup_{n=1}^{\infty}\left([b_{n+1},b_n-a_n] \times \{0\}\right) \cup \left( [0,\tfrac4{M+2}]\times\{0\} \right)\\
&= \bigcup_{n=1}^{\infty}\phi_{n,1}(L_b) \cup \bigcup_{n,i} \tau_{n,i}(L_b) \cup \sigma(L_b)\\
&\subset \bigcup_{f\in\mathcal{F}}f(K_m). \qedhere
\end{align*}
\end{proof}

\begin{proof}[{Proof of Theorem \ref{thm:1}(2)}]
We show that the attractor $K$ of the IIFS $\mathcal{F}$ is not the image of a H\"older curve. Assume for a contradiction that there exists $\alpha\geq 1$, $H>0$ and a  surjection $f:[0,1]\to K$ such that for all $x,y\in[0,1]$
\[|f(x)-f(y)|^\alpha\leq H |x-y|.\] 
For each $n\in\N$ let
\[ A_n = \bigcup_{i=1}^{M^n} \phi_{n,i}(K).\]
Note that the height of each ``tower'' $A_n$ is equal to $\log(2)/\log(n+1)$.

Recall that $M$ was chosen even and for each $n\in\N$ define $p_n = \phi_{n,M^n}((1,1))$ and $q_n = \phi_{n,\frac12M^n}((1,1))$. Following the proof of Lemma \ref{prop:curve}, for each $n\in\N$, the vertical segment $[q_n,p_n] \subset A_n$. Moreover, for any $n\in\N$, $p_n$ is the point of $A_n$ with the highest $y$-coordinate, and the $y$-coordinate of $q_n$ is $\frac12\log(2)/\log(n+1)$. Setting 
\[ B_n = \bigcup_{i=\frac12 M^n + 1}^{M^n} \phi_{n,i}(K), \]
we have $B_n \cap \overline{K\setminus B_n} = \{q_n\}$. Given that $B_n$ is connected, there exists for each $n\in\N$ an interval $I_n \subset [0,1]$ such that 
\[ \{p_n,q_n\} \subset f(I_n) \subset B_n.\]
It follows that the intervals $I_1,I_2,\dots$ are mutually disjoint and
\begin{align*}
1 \geq \sum_{n=1}^{\infty} \diam{I_n} \geq \sum_{n=1}^{\infty} H^{-1}|p_n-q_n|^{\a} = \frac{(\log{2})^{\a}}{2^{\a}H} \sum_{n=1}^{\infty} \frac{1}{(\log(n+1))^{\a}}.
\end{align*}
However, the latter series diverges and we reach a contradiction.
\end{proof}

\bibliographystyle{alpha}
\bibliography{IIFS_bibliography}

\begin{thebibliography}{FHOR15}

\bibitem[AO17]{AO}
Giovanni Alberti and Martino Ottolini.
\newblock On the structure of continua with finite length and {G}olab's
  semicontinuity theorem.
\newblock {\em Nonlinear Anal.}, 153:35--55, 2017.

\bibitem[BF23]{BF23}
Amlan Banaji and Jonathan~M. Fraser.
\newblock Intermediate dimensions of infinitely generated attractors.
\newblock {\em Trans. Amer. Math. Soc.}, 376(4):2449--2479, 2023.

\bibitem[BV21]{BV}
Matthew Badger and Vyron Vellis.
\newblock H\"{o}lder parameterization of iterated function systems and a
  self-affine phenomenon.
\newblock {\em Anal. Geom. Metr. Spaces}, 9(1):90--119, 2021.

\bibitem[FF15]{FF}
\'{A}bel Farkas and Jonathan~M. Fraser.
\newblock On the equality of {H}ausdorff measure and {H}ausdorff content.
\newblock {\em J. Fractal Geom.}, 2(4):403--429, 2015.

\bibitem[FHOR15]{FHOR}
J.~M. Fraser, A.~M. Henderson, E.~J. Olson, and J.~C. Robinson.
\newblock On the {A}ssouad dimension of self-similar sets with overlaps.
\newblock {\em Adv. Math.}, 273:188--214, 2015.

\bibitem[Hat85]{Hata}
Masayoshi Hata.
\newblock On the structure of self-similar sets.
\newblock {\em Japan J. Appl. Math.}, 2(2):381--414, 1985.

\bibitem[HU02]{HU02}
Stefan-M. Heinemann and Mariusz Urba\'{n}ski.
\newblock Hausdorff dimension estimates for infinite conformal {IFS}s.
\newblock {\em Nonlinearity}, 15(3):727--734, 2002.

\bibitem[Hut81]{Hutchinson}
John~E. Hutchinson.
\newblock Fractals and self-similarity.
\newblock {\em Indiana Univ. Math. J.}, 30(5):713--747, 1981.

\bibitem[HY61]{HY}
John~G. Hocking and Gail~S. Young.
\newblock {\em Topology}.
\newblock Addison-Wesley Publishing Co., Inc., Reading, Mass.-London, 1961.

\bibitem[JR12]{JR12}
Thomas Jordan and Michal Rams.
\newblock Increasing digit subsystems of infinite iterated function systems.
\newblock {\em Proc. Amer. Math. Soc.}, 140(4):1267--1279, 2012.

\bibitem[KZ06]{KZ06}
Marc Kesseb\"{o}hmer and Sanguo Zhu.
\newblock Dimension sets for infinite {IFS}s: the {T}exan conjecture.
\newblock {\em J. Number Theory}, 116(1):230--246, 2006.

\bibitem[Mau95]{M}
R.~Daniel Mauldin.
\newblock Infinite iterated function systems: theory and applications.
\newblock In {\em Fractal geometry and stochastics ({F}insterbergen, 1994)},
  volume~37 of {\em Progr. Probab.}, pages 91--110. Birkh\"{a}user, Basel,
  1995.

\bibitem[MMU01]{MVU01}
R.~Daniel Mauldin, Volker Mayer, and Mariusz Urba\'{n}ski.
\newblock Rigidity of connected limit sets of conformal {IFS}s.
\newblock {\em Michigan Math. J.}, 49(3):451--458, 2001.

\bibitem[MSU09]{MSU09}
R.~D. Mauldin, T.~Szarek, and M.~Urba\'{n}ski.
\newblock Graph directed {M}arkov systems on {H}ilbert spaces.
\newblock {\em Math. Proc. Cambridge Philos. Soc.}, 147(2):455--488, 2009.

\bibitem[MU96]{MU}
R.~Daniel Mauldin and Mariusz Urba\'{n}ski.
\newblock Dimensions and measures in infinite iterated function systems.
\newblock {\em Proc. London Math. Soc. (3)}, 73(1):105--154, 1996.

\bibitem[MU99]{MU99}
R.~Daniel Mauldin and Mariusz Urba\'{n}ski.
\newblock Conformal iterated function systems with applications to the geometry
  of continued fractions.
\newblock {\em Trans. Amer. Math. Soc.}, 351(12):4995--5025, 1999.

\bibitem[MU02]{MU02}
R.~Daniel Mauldin and Mariusz Urba\'{n}ski.
\newblock Fractal measures for parabolic {IFS}.
\newblock {\em Adv. Math.}, 168(2):225--253, 2002.

\bibitem[Rem98]{Remes}
Marko Remes.
\newblock H\"{o}lder parametrizations of self-similar sets.
\newblock {\em Ann. Acad. Sci. Fenn. Math. Diss.}, (112):68, 1998.

\bibitem[RGU16]{RGU16}
Lasse Rempe-Gillen and Mariusz Urba\'{n}ski.
\newblock Non-autonomous conformal iterated function systems and {M}oran-set
  constructions.
\newblock {\em Trans. Amer. Math. Soc.}, 368(3):1979--2017, 2016.

\bibitem[Sch94]{Sch94}
Andreas Schief.
\newblock Separation properties for self-similar sets.
\newblock {\em Proc. Amer. Math. Soc.}, 122(1):111--115, 1994.

\bibitem[Sch96]{Sch96}
Andreas Schief.
\newblock Self-similar sets in complete metric spaces.
\newblock {\em Proc. Amer. Math. Soc.}, 124(2):481--490, 1996.

\bibitem[SW15]{SW15}
S.~Seuret and B.-W. Wang.
\newblock Quantitative recurrence properties in conformal iterated function
  systems.
\newblock {\em Adv. Math.}, 280:472--505, 2015.

\bibitem[UZ02]{UZ02}
Mariusz Urba\'{n}ski and Anna Zdunik.
\newblock Hausdorff dimension of harmonic measure for self-conformal sets.
\newblock {\em Adv. Math.}, 171(1):1--58, 2002.

\end{thebibliography}

\end{document}